\documentclass[12pt,twoside]{amsart}
\usepackage{amsfonts,amsmath}
\usepackage[all,knot,arc]{xy}
\usepackage[dvips]{graphicx}
\usepackage{psfrag}

\def\endpf{\relax\ifmmode\expandafter\endproofmath\else
  \unskip\nobreak\hfil\penalty50\hskip.75em\hbox{}\nobreak\hfil\bull
  {\parfillskip=0pt \finalhyphendemerits=0 \bigbreak}\fi}
\def\bull{\vbox{\hrule\hbox{\vrule\kern3pt\vbox{\kern6pt}\kern3pt\vrule}\hrule}}

\newtheorem{defn}{Definition}[section]
\newtheorem{lemma}[defn]{Lemma}

\newtheorem{theorem}[defn]{Theorem}

\newtheorem{remark}[defn]{Remark}
\newtheorem{proposition}[defn]{Proposition}

\newtheorem{maintheorem}{Theorem}
\newtheorem{maincor}[maintheorem]{Corollary}
\newtheorem{example}[defn]{Example}
\newtheorem{maindefn}{Definition}

\newcommand{\zz}{{\mathbb Z}}

\newcommand{\qq}{{\mathbb Q}}

\newcommand{\ozsvath}{Ozsv\'{a}th}
\newcommand{\szabo}{Szab\'{o}}

\newcommand{\spin}{\ifmmode{\rm Spin}\else{${\rm spin}$\ }\fi}
\newcommand{\spinc}{\ifmmode{{\rm Spin}^c}\else{${\rm spin}^c$\ }\fi}

\newcommand{\spincs}{\mathfrak s}

\newcommand{\lk}{{\rm lk}}

\newcommand{\calc}{\mathcal{C}}
\newcommand{\lc}{\mathcal{L}}
\newcommand{\tlc}{\widetilde{\mathcal{L}}}
\newcommand{\tl}{\tilde{l}}
\newcommand{\Top}{\mathrm{TOP}}
\newcommand{\lctop}{\lc_\Top}
\newcommand{\tlctop}{\tlc_\Top}
\newcommand{\mobius}{M\"{o}bius}
\newcommand{\calctop}{\calc_\Top}
\newcommand{\calf}{\mathcal{F}}
\newcommand{\tcalf}{\widetilde{\mathcal{F}}}

\newcommand{\calr}{\mathcal{R}}

\newcommand{\calh}{\mathcal{H}}

\newcommand{\caln}{\mathcal{N}}
\newcommand{\tcaln}{\widetilde{\mathcal{N}}}

\newcommand{\ds}{\displaystyle}

\newenvironment{narrow}[2]{%
 \begin{list}{}{%
  \setlength{\topsep}{0pt}%
  \setlength{\leftmargin}{#1}%
  \setlength{\rightmargin}{#2}%
  \setlength{\listparindent}{\parindent}%
  \setlength{\itemindent}{\parindent}%
  \setlength{\parsep}{\parskip}%
 }%
\item[]}{\end{list}}

\newif\ifpic
\picfalse    
\pictrue   

\DeclareMathOperator\fr{Fr}

\begin{document}

\title{Concordance groups of links}
\author{Andrew Donald and Brendan Owens}
\date{\today}
\thanks{}

\begin{abstract}
We define a notion of concordance based on Euler characteristic, and show that it gives rise to a concordance group $\lc$ of links in $S^3$, which has the concordance group of knots as a direct summand with infinitely generated complement.  We consider variants of this using oriented and nonoriented surfaces as well as smooth and locally flat embeddings.
\end{abstract}

\maketitle

\pagestyle{myheadings} \markboth{ANDREW DONALD AND BRENDAN OWENS}
{CONCORDANCE GROUPS OF LINKS}


\section{Introduction}
\label{sec:intro}
A knot $K$ in $S^3$ is \emph{slice} if it bounds a smoothly embedded disk  $\Delta$ in the four-ball; it is topologically slice if it bounds a locally flat embedded disk.
Two oriented knots $K_0$, $K_1$ are \emph{concordant} if the connected sum $-K_0\#K_1$
of one with the reverse mirror of the other is slice.  This is an equivalence relation, and Fox and Milnor showed that the set of equivalence classes forms a group 
$\calc$ under connected sum \cite{fm1,fm2}.  Our goal in this paper is to generalise this construction in a natural way to links.  The starting point is 
Lisca's work  \cite{lisca1,lisca2} on two-bridge links and lens spaces, as well as earlier work of Florens \cite{florens1,florens2}.  These indicate that the following is a natural  generalisation of sliceness to links. 

\begin{figure}[htbp]
\centering
\ifpic
\includegraphics[width=14cm]{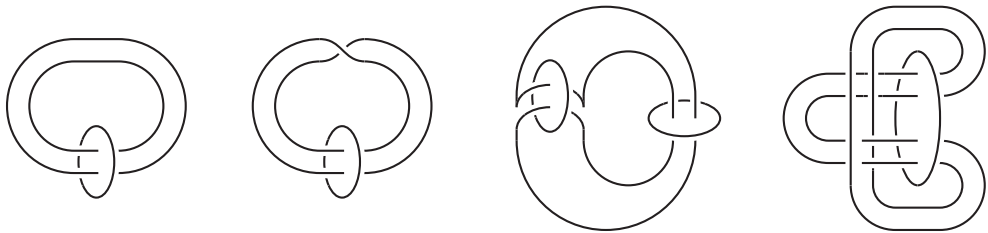}
\else \vskip 5cm \fi
\begin{narrow}{0.3in}{0.3in}
\caption{
\bf{Some links bounding $\chi=1$ surfaces in the four-ball: connected sum of two Hopf links, $(2,4)$-torus link, Borromean rings and connected sum of Hopf and Whitehead links.
}}
\label{fig:ribbonF}
\end{narrow}
\end{figure}

\begin{maindefn}
\label{def:chislice1}
A link $L$ in $S^3$ is $\chi$-slice\footnote{This is called \emph{geometrically bordant} in \cite{florens2}, in the case that the surface $F$ is orientable.}
 if $L$ bounds a smoothly properly embedded surface
$F$ in $D^4$ without closed components, and with $\chi(F)=1$.  If $L$ is oriented we require $F$ to be compatibly oriented.
\end{maindefn}

Some examples of $\chi$-slice links are shown in Figure \ref{fig:ribbonF}.  Note we do not in general require that $F$ is connected or oriented.  Observe however that if $L$ is a knot then $F$ is a disk, so this notion of sliceness coincides with the usual one.

The set of oriented knots is an abelian monoid under connected sum, together with an involution $K\mapsto -K$.  We wish to endow the set of links with a compatible monoid with involution structure.  We use the term \emph{partly oriented link} to denote a link with a marked oriented component and the remaining components unoriented, and the term \emph{marked oriented link} to denote an oriented link with a marked component.  Connected sum is well-defined for these sets of links using the marked components.  We define $-L$ to be the mirror of $L$, with orientations reversed.

\begin{figure}[htbp]
{\centering
\ifpic
\psfrag{1}{\scriptsize$H=$}
\psfrag{2}{\scriptsize$L_1=$}
\psfrag{3}{\scriptsize$\tilde{H}=$}
\psfrag{4}{\scriptsize$\tilde{L}_1=$}
\includegraphics[width=12cm]{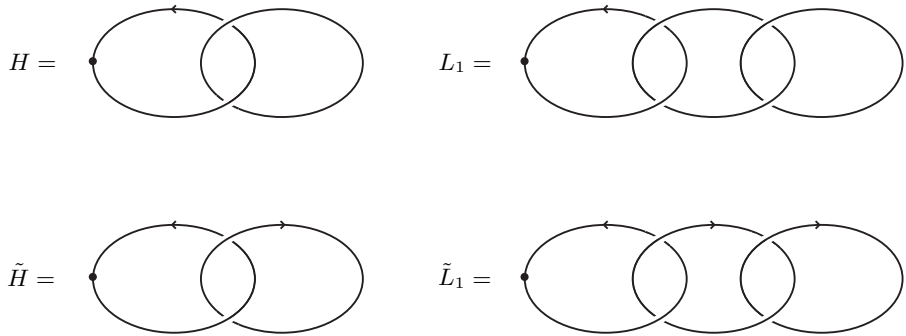}
\else \vskip 5cm \fi
\begin{narrow}{0.3in}{0.3in}
\caption{
\bf{Two partly oriented links $H$ and $L_1$ and two marked oriented links $\tilde{H}$ and $\tilde{L}_1$.}}
\label{fig:HL1}
\end{narrow}
}
\end{figure}

This gives the following commutative diagram of monoids with involution, where the vertical arrow is the map given by forgetting the orientation on nonmarked components.

\begin{equation}
\label{knotlink}
\leavevmode
\begin{xy} 
(0,10)*+{\text{\{Oriented knots\}}}="k";
(60,0)*+{\text{\{Partly oriented links\}}}="p";
(60,20)*+{\text{\{Marked oriented links\}}}="m";
{\ar@{^{(}->} "k"; "p"};
{\ar@{^{(}->} "k"; "m"};
{\ar@{->>} "m"; "p"};
\end{xy}
\end{equation}

We wish to quotient each of these link monoids by a suitable submonoid such that the maps in \eqref{knotlink} induce inclusions of the knot concordance group $\calc$ into two different concordance groups of links.  Roughly speaking we would like to quotient out by $\chi$-slice links, but it turns out we must be a little more careful in order to get an equivalence relation, and to preserve connected sums.

\begin{maindefn}
\label{def:chiconc}
Let $L_0$ and $L_1$ be partly oriented or marked oriented links.  We say $L_0$ and $L_1$ are $\chi$-concordant, written $L_0\sim L_1$, if $-L_0\#L_1$ bounds 
a smoothly properly embedded surface $F$ in $D^4$ such that
\begin{itemize}
\item $F$ is a disjoint union of one disk together with annuli and M\"{o}bius bands;
\item the boundary of the disk component of $F$ is the marked component of $-\!L_0\,\#L_1$;
\item in the marked oriented case, we require $F$ to be oriented and $-\!L_0\,\#L_1$ to be the oriented boundary of $F$.
\end{itemize}
\end{maindefn}

Note that $\chi$-concordance agrees with the usual definition of smooth concordance if $L_0$ and $L_1$ are both knots.  Also $L_0\sim L_1$ implies that
$-L_0\#L_1$ is $\chi$-slice, but the converse does not hold.  We will elaborate on this point in Section \ref{sec:alg}.  We have the following basic results.

\begin{maintheorem}
\label{thm:mainthm}
The set of $\chi$-concordance classes of partly oriented links forms an abelian group
$$\lc\cong\calc\oplus\lc_0$$
under connected sum which contains the smooth knot concordance group $\calc$ as a direct summand.
The inclusion $\calc\hookrightarrow\lc$ is induced by the inclusion of oriented knots into partly oriented links.

The complement $\lc_0$ of $\calc$ in $\lc$ contains a $\zz/2$ direct summand and a $\zz^\infty\oplus(\zz/2)^\infty$ subgroup.
\end{maintheorem}

\begin{maintheorem}
\label{thm:or}
The set of $\chi$-concordance classes of marked oriented links forms an abelian group
$$\tlc\cong\calc\oplus\tlc_0$$
under connected sum which contains the smooth knot concordance group $\calc$ as a direct summand
(with  $\calc\hookrightarrow\tlc$ induced by the inclusion of oriented knots into marked oriented links).
Forgetting orientations on nonmarked components induces a surjection $\tlc\to\lc$.  In other words, we have the following group
homomorphisms induced by \eqref{knotlink}:

\begin{center}
\leavevmode
\begin{xy} 
(0,10)*+{\calc\ }="k";
(20,0)*+{\lc}="p";
(20,20)*+{\tlc}="m";
{\ar@{^{(}->} "k"; "p"};
{\ar@{^{(}->} "k"; "m"};
{\ar@{->>} "m"; "p"};
\end{xy}
\end{center}

The complement $\tlc_0$ of $\calc$ in $\tlc$ contains a $\zz\oplus\zz/2$ direct summand and a $\zz^\infty$ subgroup.
\end{maintheorem}

We find that many familiar tools from the study of knot concordance are  applicable to these link concordance groups.
Let $\caln<\lc$ and $\tcaln<\tlc$ be the subgroups consisting of classes represented by links with nonzero determinant.
More generally for $\omega\in S^1\setminus\{1\}$ we let $\tcaln_\omega<\tlc$ be the subgroup
of links with vanishing Levine-Tristram nullity $n_\omega$, so that $\tcaln=\tcaln_{-1}$.  The following contains a collection of invariants that may be used
in studying $\lc$ and $\tlc$.

\begin{maintheorem}
\label{thm:invts}
Taking total linking number with the marked component gives homomorphisms
\begin{align*}
l:\lc&\to\zz/2,\\
\tl:\tlc&\to\zz.
\end{align*}
Taking double branched covers gives group homomorphisms
\begin{align*}
\calf:\caln&\to\Theta^3_\qq,\\
\tcalf:\tcaln&\to\Theta^3_{\qq,\spin}
\end{align*}
to the rational homology cobordism group (resp., spin rational homology cobordism group) of (spin) rational homology three-spheres.
Link signature and the \ozsvath-\szabo\ correction term of a spin structure in the double branched cover, give homomorphisms
$$\sigma, \delta:\tcaln\to\zz.$$
The sum $\sigma+\delta$
is divisible by 8 for all links with nonzero determinant and is zero for alternating links.

For each prime-power root of unity $\omega\in S^1\setminus\{1\}$, the Levine-Tristram signature $\sigma_\omega$ gives a homomorphism
from $\tcaln_\omega$ to the integers.
\end{maintheorem}

We note also that the Fox-Milnor condition on the Alexander polynomial of slice knots extends to links bounding certain surfaces of Euler characteristic one in the four-ball \cite{orevkov,florens2}.  We are currently investigating, jointly with Stefan Friedl, possible extensions of the Fox-Milnor obstruction.

We also consider topological link concordance groups $\lctop$ and $\tlctop$,  where we replace smooth with locally flat embeddings.  We state topological versions of Theorems
\ref{thm:mainthm} and \ref{thm:or} in Section \ref{sec:top} but for now we note some examples, based on work of Davis and Cha-Kim-Ruberman-Strle, distinguishing the two categories.

\begin{maintheorem}
\label{thm:TOPeg}
Let $K$ be an alternating knot with negative signature (for example the right handed trefoil), and let
$C$ be a knot with Alexander polynomial one and  $\delta(C)\ne0$.  
The partly-oriented links $L_2\#H$ and $L_3\#H$shown in Figure \ref{fig:DiffTop} are trivial in $\lc_{\Top}$ and nontrivial in $\lc$.

Orienting all components of $L_2\#H$ and $L_3\#H$ results in marked oriented links which are trivial in $\tlc_{\Top}$ and nontrivial in $\tlc$, under the same hypotheses on $K$ and $C$.
\end{maintheorem}

\begin{figure}[htbp]
{\centering
\ifpic
\psfrag{K}{\scriptsize$K$}
\includegraphics[width=14cm]{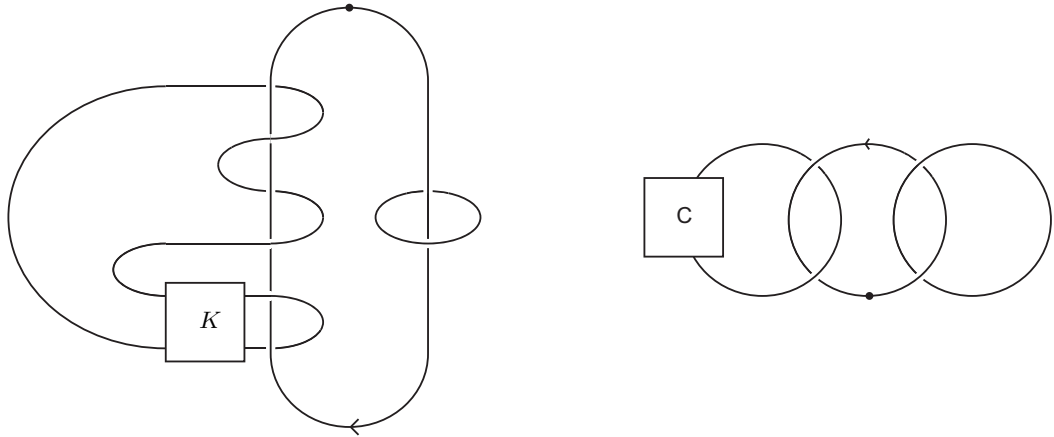}
\else \vskip 5cm \fi
\begin{narrow}{0.3in}{0.3in}
\caption{
\bf{Partly oriented links $L_2\# H$ and $L_3\# H$.  The band shown passing through the box marked $K$ is tied in the knot $K$ with zero framing (cf.~\cite{ckrs}).}}
\label{fig:DiffTop}
\end{narrow}
}
\end{figure}

In \cite{lisca1}, Lisca proved the slice-ribbon conjecture for two-bridge knots.  His results also covered the case of two-bridge links.
Combining his work with an observation in this paper yields the following slice-ribbon result for two-bridge links.

\begin{maincor}
\label{cor:2bridge}
Let $S(p,q)$ be a two-bridge link. If $F$ is a smoothly properly embedded surface in $D^4$ with $\chi(F)=1$ and no closed components, bounded by $S(p,q)$, then the link also bounds a ribbon embedding of $F$.
\end{maincor}

\noindent{\bf Related work.}
The problem of knots and links bounding non-orientable surfaces in the four-ball has recently been considered by Gilmer-Livingston \cite{gl} who study connected nonorientable surfaces bounded by a knot.  Orevkov and Florens \cite{orevkov,florens1,florens2} have considered the problem of links bounding orientable surfaces of Euler characteristic one.  Baader \cite{baader} has defined a notion of cobordism distance between oriented links such that $\chi$-sliceness is equivalent to cobordism distance zero from the unknot.

Hosokawa \cite{hosokawa} gave a different definition of a concordance group $\calh$ of links containing $\calc$ as a direct summand, following a suggestion of Fox.  Hosokawa also showed that
$$\calh\cong\calc\oplus\zz,$$
in contrast to our results.

\noindent{\bf Acknowledgements.}
This paper was inspired by Lisca's work on two-bridge links and lens spaces.  We are grateful to Stefan Friedl, Cameron Gordon, Matt Hedden, Slaven Jabuka, Paul Kirk, Paolo Lisca, Swatee Naik and Jake Rasmussen for helpful comments and conversations.


\section{A link concordance group using smooth surfaces in $D^4$}
\label{sec:pf}

In this section we prove Theorem \ref{thm:mainthm}.  We show that $\chi$-concordance gives rise to a group $\lc$ which contains the classical knot concordance group as a direct summand,  and we describe some group homomorphisms from $\lc$.  We begin by describing $\chi$-concordance using embedded surfaces in the cylinder $S^3\times [0,1]$.


\begin{lemma}
\label{lem:cyl}
Partly oriented links $L_0$, $L_1$ are $\chi$-concordant if and only if there exists a smoothly properly embedded surface $F_0$ in $S^3\times [0,1]$ satisfying 
\begin{itemize}
\item $F_0$ is a disjoint union of annuli, including one oriented annulus $A$, and M\"{o}bius bands;
\item $F_0\cap S^3\times\{i\}=L_i\times\{i\},\quad i=0,1;$
\item $\partial A=\vec{K}_1\times\{1\}\cup \vec{K}_0\!^r\times\{0\},$
where $\vec{K}_i$ is the oriented component of $L_i$ and $\vec{K}_0\!^r$ denotes the knot $\vec{K}_0$ with the opposite orientation.
\end{itemize}
\end{lemma}
\begin{proof}
This follows from Definition \ref{def:chiconc} as in the standard knot situation: one passes between $(D^4,F)$ and $(S^3\times [0,1],F_0)$ by drilling out an arc of $A$ or attaching a $(3,1)$-handle pair.
\end{proof}

\begin{lemma}
\label{lem:equiv}
$\chi$-concordance is an equivalence relation.
\end{lemma}
\begin{proof}
For any partly oriented link $L$, $-L\#L$ is $\chi$-nullconcordant ($\chi$-concordant to the unknot) by the usual argument for knots.  That is to say,
the connected sum may be arranged so that it is symmetric about a plane containing two points on the oriented component $-\!\vec{K}\#\vec{K}$.  Rotating the link about this plane in four-dimensional half-space (which is diffeomorphic to the punctured four-ball) yields a surface $F$ which is a disjoint union of a disk bounded by $-\!\vec{K}\#\vec{K}$ and one annulus for each unoriented component of $L$.

Symmetry is immediate from Definition \ref{def:chiconc}: applying an orientation reversing diffeomorphism to the four-ball takes a surface bounded by $-L_0\#L_1$ to one bounded by $-L_1\#L_0$.  Transitivity follows by composing the cobordisms $F_0$ from Lemma \ref{lem:cyl}.  Any resulting closed components may be discarded.
\end{proof}

\begin{lemma}
\label{lem:groupsum}
The set of $\chi$-concordance classes of partly oriented links is an abelian group $\lc$ under connected sum, which contains the knot concordance group as a direct summand.
The direct complement $\lc_0$ consists of equivalence classes of partly oriented links $L$ whose oriented component $\vec{K}$ is a slice knot.
An isomorphism
$$\lc\stackrel{\cong}{\longrightarrow}\calc\oplus\lc_0$$
is given by
$$[L]\mapsto([\vec{K}],[-\!\vec{K}\#L]).$$
\end{lemma}
\begin{proof}
Connected sum is well-defined, abelian, and associative for partly oriented links, by a variant of the usual proof for knots (see e.g.~\cite{bz}).
Suppose that $L$, $L_0$ and $L_1$ are partly oriented links, and that $L_0\sim L_1$.  Let $F_0$ be the cobordism  in $S^3\times [0,1]$ between $L_0$ and $L_1$, as in Lemma \ref{lem:cyl}, with oriented annulus component $A$.  Taking connected sum ``along the annulus" $A$ shows that $L_0\# L\sim L_1\# L$.  It follows that connected sum gives a well-defined operation on $\lc$.
The identity is given by the class of the unknot and the inverse of $[L]$ is $[-L]$.  The inclusion of oriented knots into partly oriented links induces a monomorphism $\calc\to\lc$ since $\chi$-concordance of knots is the same as knot concordance.  A splitting homomorphism is given by
$$[L]\longmapsto[\vec{K}],$$
taking the $\chi$-concordance class of a partly oriented link to the concordance class of its oriented component.  It follows that $[L]$ is in the direct complement $\lc_0$ if and only if the oriented component of $L$ is slice.  For any partly oriented link $L$ with oriented component $\vec{K}$ we have $[-\vec{K}\#L]\in \lc_0$ and
$$L\sim\vec{K}\#\!-\!\vec{K}\#L$$
by associativity, from which the stated isomorphism follows.
\end{proof}

We obtain a $\zz/2$-valued homomorphism from $\lc$ using mod 2 linking numbers via the following lemma.

\begin{lemma}
\label{lem:F1F2}
Let $L$ be a link in $S^3$ bounding a smoothly properly embedded surface $F$ in $D^4$, and suppose that $F=F_1\sqcup F_2$ is a disjoint union.  This gives a decomposition of $L$ into $L_1\sqcup L_2$, where $L_i=\partial F_i$.  Then the total mod 2 linking number of $L_1$ with $L_2$ is zero, i.e.
$$\ds\sum_{K_i \,\mathrm{ in }\, L_i}\lk(K_1,K_2)\equiv0\pmod2.$$
\end{lemma}
\begin{proof}
We may assume the radial distance function $r$ on $D^4$ restricts to give
a Morse function on $F$ with values in $[0.5,1]$.  Let $(F_1)_t$ (respectively, $(F_2)_t$) be the level set of $r$ restricted to $F_1$ (resp. $F_2$) for each $t$, so that $(F_i)_1=L_i$ and $(F_i)_{0.4}$ is empty.  The mod 2 sum $s(t)$ of the linking numbers of each component in $(F_1)_t$ with each component in $(F_2)_t$ is constant with respect to $t$ since this sum does not change at a maximum or minimum and changes by an even number at a saddle point of $F$.  Thus,
$$\ds\sum_{K_i \mathrm{in} L_i}\lk(K_1,K_2)=s(1)\equiv s(0.4)=0\pmod2.$$
\end{proof}

It follows from Lemma \ref{lem:F1F2} 
that we get a homomorphism
$$l:\lc\longrightarrow\zz/2$$
by taking $$l([L])=\ds\sum_{K'\ne \vec{K}}\lk(\vec{K},K'),$$
where $\vec{K}$ is the oriented component of $L$.
The Hopf link $H$ (with one marked oriented component) satisfies $H=-H$ and $l(H)=1$, and thus generates a $\zz/2$ summand of $\lc_0$.

\begin{example}
\label{eg:HW}
Figure \ref{fig:ribbonF} shows a ribbon immersed disk disjoint union annulus bounded by the connected sum of the Hopf and Whitehead links.  However the Hopf link has mod 2 linking number $l=1$ while the Whitehead link has vanishing $l$.  It follows that their sum is nontrivial in $\lc$.  This illustrates a subtlety of the definition of $\lc$: the connected sum of the Hopf and Whitehead links is a partly oriented link which bounds a surface $F$ in the four-ball with $\chi=1$.  However it does not bound any such surface with its oriented component bounding a disk component of $F$.
\end{example}


We recall that the group $\Theta^3_\qq$ consists of smooth rational homology cobordism classes of rational homology three-spheres under connected sum.  Two rational homology three-spheres $Y_0$ and $Y_1$ are rational homology cobordant if $-Y_0\# Y_1$ bounds a rational homology four-ball, or equivalently if $-Y_0$ and $Y_1$ cobound a rational homology $S^3\times [0,1]$.

We next show that taking double branched covers yields a group homomorphism
$$\calf:\caln\to\Theta^3_\qq,$$
where $\caln$ is the subgroup of $\lc$ consisting of classes represented by links with nonzero determinant.
This is a consequence of the following proposition, which is proved in Section \ref{sec:2covers}.  A proof was given by Lisca in \cite{lisca1} for the case of ribbon embedded surfaces.

\begin{proposition}
\label{prop:2covers}
Let $L$ be a link in $S^3$ with nonzero determinant which bounds a smoothly (or topologically locally flat) properly embedded surface
$F$ in $D^4$ without closed components, and with $\chi(F)=1$.  Then the double cover of $D^4$ branched along $F$ is a smooth (or topological) rational homology four-ball.
\end{proposition}

One consequence of Proposition \ref{prop:2covers} is that the determinant of any $\chi$-nullconcordant link is a square.

It remains to be seen that $\lc_0$ contains a $\zz^\infty\oplus(\zz/2)^\infty$ subgroup.  

\begin{proposition}
\label{prop:2torsgroup}
The two-bridge links $\{S(q^2+1,q)\,|\,q\text{ odd}\}$ generate a $(\zz/2)^\infty$ subgroup of $\lc_0$.
\end{proposition}
\begin{proof}
Each partly oriented link $L=S(q^2+1,q)$ for $q$ odd satisfies $L=-L$ and therefore has order one or two in $\lc$; since $q^2+1$ is not a square the order is two, by Proposition \ref{prop:2covers}.  The components of a two-bridge link are one-bridge and hence unknots, thus a two-bridge link represents an element of $\lc_0$.  We could appeal to Lisca's results \cite{lisca2} to see that there are no other relations among these links but there is an easier argument using determinants.

We will show that the subgroup of $\lc_0$ generated by $\{S(q^2+1,q)\,|\,q\text{ odd}\}$ is infinitely generated and hence is isomorphic to $(\zz/2)^\infty$.  Suppose we have some finite subset
$\{S(q_i^2+1,q_i)\}$.  Choose a prime $p$ congruent to 1 modulo 4 which does not divide $q_i^2+1$ for each $i$.  Then there exists an odd positive  $q<p$ with $q^2+1$ divisible by $p$ but not by $p^2$.  It follows, again using Proposition \ref{prop:2covers}, that $S(q^2+1,q)$ is not in the subgroup of $\lc_0$ generated by $\{S(q_i^2+1,q_i)\}$.
\end{proof}

\begin{proposition}{(Corollary of \cite[Theorem 1.1]{lisca2})}
\label{prop:lens}
The subgroup of the rational homology cobordism group of rational homology 3-spheres $\Theta^3_\qq$ generated by lens spaces is infinitely generated.
In particular the set
$$\{L(2k,1)\,|\,k>2\}$$
is independent in $\Theta^3_\qq$.
\end{proposition}

\begin{proof}
This follows from \cite[Theorem 1.1]{lisca2} since for $k>2$, $L(2k,1)$ is not contained in any of Lisca's families $\calr$ or $\calf_n$.
\end{proof}

The lens space $L(2p,q)$ is the double branched cover of the two-bridge link $S(2p,q)$ each of whose components is an unknot.
Combining Propositions \ref{prop:2covers} and \ref{prop:lens} we see that the two-bridge links
$$\{S(2k,1)\,|\,k>2\}$$
generate a $\zz^\infty$ subgroup of $\lc_0$.  (An argument with determinants can be used to show the subgroup these links generate is not finitely generated, without appealing to \cite{lisca2}.) This completes the proof of Theorem \ref{thm:mainthm}.


\section{Using smooth oriented surfaces}
\label{sec:or}

In this section we prove Theorem \ref{thm:or}, and complete the proof of Theorem \ref{thm:invts}.

We use the term \emph{marked oriented link} for an oriented link in $S^3$ with one marked component.  The marked components are used when taking connected sums.
The reverse mirror of $L$, preserving the marked component, is denoted $-L$.
Marked oriented links $L_0$ and $L_1$ are $\chi$-concordant if $-L_0\#L_1$ bounds an oriented smoothly properly embedded disjoint union of a disk with annuli in $D^4$,
with the marked component bounding the disk.  Modifying Lemma \ref{lem:cyl}, this is equivalent to $L_0\!^r\times\{0\}$ and $L_1\times\{1\}$ being the oriented boundary of a disjoint union $F_0$ of properly embedded annuli in $S^3\times[0,1]$, with one component of $F_0$ connecting the marked components.  It follows that $L_0$ and $L_1$ have the same number of components modulo two.

 Lemmas \ref{lem:equiv} and \ref{lem:groupsum} can be restated for the case of marked oriented links and the same proofs apply.  The group of $\chi$-concordance classes of marked oriented links
is denoted $\tlc$.  Comparing the definitions  we see that if marked oriented links $L$ and $L'$ represent the same class in $\tlc$ then the embedded surface giving rise to the $\chi$-concordance also gives rise to a $\chi$-concordance between the partly oriented links obtained from $L$ and $L'$ by forgetting orientations on nonmarked components.  This forgetful map also commutes with connected sum and so gives rise to a surjection from $\tlc$ to $\lc$.

\begin{lemma}
\label{lem:Zlinking}
Let $L$ be a $\chi$-nullconcordant marked oriented link with marked component $K$.  Then $\ds\sum\lk(K,K')=0$, where the sum is taken over all components $K'\ne K$ of $L$.
\end{lemma}
\begin{proof}
This follows from a modification of the proof of Lemma \ref{lem:F1F2}, taking $F$ to be a surface in $D^4$ witnessing the $\chi$-nullconcordance, with $F_1$ the disk component bounded by $K$.  In the oriented case the sum of linking numbers between the level set of $F_1$ and that of $F_2$ does not change at any critical point of $r|_F$.
\end{proof}

It follows from Lemma \ref{lem:Zlinking} that the total linking number with the marked component gives a homomorphism
$$\tilde{l}:\tlc\to\zz$$
which is a lift of $l:\lc\to\zz/2$.  A homomorphism $\mu$ to $\zz/2$ is given by taking one plus the number of components of a link modulo two.

The marked oriented (positive) Hopf link $H$ has $l=\mu=1$ and the marked oriented two component unlink $U$ has $l=0$, $\mu=1$ and order two in $\tlc$.   Thus these two links generate a $\zz\oplus\zz/2$ summand of $\tlc_0$.

Let $\omega\in S^1\setminus\{1\}$ be a prime-power root of unity.  The Levine-Tristram signature $\sigma_\omega$ and nullity $n_\omega$ are
defined to be the signature and nullity of $(1-\bar\omega)M+(1-\omega)M^T$ where $M$ is a Seifert matrix for $L$.
It follows that both of these invariants are additive under connected sum of marked oriented links.  The nullity is invariant and the signature changes sign under $L\mapsto -L$.
We let $\tcaln_\omega$ be the subgroup consisting of elements with a representative with
zero Levine-Tristram nullity $n_\omega$ (or equivalently, with $\omega$ not a root of the one-variable Alexander polynomial).

\begin{lemma}
\label{lem:sig}
Let $L$ be an oriented link with $n_\omega(L)=0$.  If $L$ is $\chi$-slice then $\sigma_\omega(L)=0$.  It follows that the Levine-Tristram signature gives
a homomorphism
$$\sigma_\omega:\tcaln_\omega\to\zz.$$
\end{lemma}
\begin{proof}
The vanishing of the Levine-Tristram signature for a $\chi$-slice link with $n_\omega(L)=0$ follows directly from the Murasugi-Tristram inequality, see \cite[Theorem 2.27]{tristram}, also \cite{kt,gilmer,florens2,cimasoniflorens}.
\end{proof}

In \cite[\S 2.1]{turaev}, Turaev constructed a bijection from the set of quasiorientations (orientations up to overall reversal) on a link $L$ in $S^3$ to the set of spin structures on the double-branched cover $\Sigma_2(S^3,L)$.  In the following Proposition, the proof of which closely follows \cite{turaev}, we extend this map to the case of an orientable surface in the four-ball.

\begin{proposition}
\label{prop:spin}
Let $F$ be an oriented smoothly 
properly embedded surface in $D^4$ and let $N$ be the double cover of $D^4$ branched along $F$.  
There is a natural bijective correspondence between quasiorientations of $F$ and spin structures on $N$.
The spin structure on $\partial N$ determined by the induced orientation on the link
$L=\partial F\subset S^3$ admits an extension over $N$, which is unique if $F$ has no closed components.
\end{proposition}
\begin{proof}

Let $F_1,\dots,F_m$ be the components of $F$ and let $\mu_i$ be an oriented meridian of $F_i$.  The first homology of $D^4\setminus F$ is freely generated over $\zz$ by the meridians.
Let $\gamma:H_1(D^4\setminus F)\to\zz$ be the homomorphism taking each meridian to $1$.  Denoting by $\tilde{F}$ the preimage of $F$ in $N$, the double covering
$$\pi:N\setminus\tilde{F}\to D^4\setminus F$$
corresponds to the mod 2 reduction of $\gamma$.  Hence a loop $l$ in $D^4\setminus F$ lifts to $N\setminus \tilde{F}$ if and only if $\gamma([l])$ is even.
Thus
\begin{equation}
\label{eqn:h}
h=\dfrac{\gamma\circ\pi_*}{2}\pmod 2
\end{equation}
is an element of $H^1(N\setminus\tilde{F};\zz/2)\cong \mathrm{Hom}(H_1(N\setminus\tilde{F}),\zz/2)$.

Choose a meridional disk $D_i$ for each surface component $F_i$, with $\partial D_i=\mu_i$.  Let $\tilde{D}_i$ and $\tilde{\mu}_i$ be the preimages of $D_i$ and $\mu_i$ in $N$.
A spin structure $\spincs$ on $N\setminus\tilde{F}$ extends uniquely over $N$ if and only if its restriction to $\tilde{\mu}_i$ extends over $\tilde{D}_i$ for each $i$.

The frame bundle $\fr(S^1)$ of $S^1$ is a copy of $S^1$.   There are two spin structures on the circle corresponding to the two double covers of $\fr(S^1)$.  The unique spin structure on $D^2$ restricts to the nontrivial spin structure on $S^1$.  The pullback of this nontrivial spin structure to the nontrivial double cover of $S^1$ is trivial.

Let $\tilde{\spincs}$ be the spin structure on $N\setminus\tilde{F}$ pulled back from $D^4\setminus F$.  The spin structure on $D^4$ restricts to the nontrivial spin structure on each $\mu_i$, which pulls back to the trivial spin structure on $\tilde{\mu}_i$.  Thus $\tilde{\spincs}$ does not extend over $\tilde{D}_i$.  However since $h(\tilde{\mu}_i)=1$ for each $i$,
 the spin structure $\tilde{\spincs}+h$ extends over $N$.  The proof that this gives a bijection between quasiorientations of $F$ and $\spin(N)$ follows \cite{turaev}: given two components $F_i$ and $F_j$ of $F$, changing orientation on just one of them will change the value of $h$ on $\tilde{\mu}_i+\tilde{\mu}_j$.  This proves injectivity, and surjectivity follows since the order of $H^1(N;\zz/2)$ is $2^{m-1}$ (see for example \cite[Theorem 1]{lw}).

The spin structure on $Y=\Sigma_2(S^3,L)$ described by Turaev in \cite{turaev} is defined in exactly the same way: $\gamma_L:H_1(S^3\setminus L)\to\zz$ takes each oriented meridian to $1$
and this defines $h_L\in H^1(Y \setminus \tilde{L};\zz/2)$ as in \eqref{eqn:h}.  The spin structure pulled back from $S^3$, twisted by $h_L$, extends uniquely over $Y$.
It is clear that this is the restriction of $\tilde{\spincs}+h$.  The uniqueness of the extension in the case that $F$ has no closed components follows since restriction of quasiorientations from $F$  
 to $\partial F$ is injective.
\end{proof}

Recall that the group $\Theta^3_{\qq,\spin}$ consists of smooth spin rational homology cobordism classes of spin rational homology three-spheres under connected sum.  Two spin rational homology three-spheres $Y_0$ and $Y_1$ are spin rational homology cobordant if $-Y_0\# Y_1$ bounds a spin rational homology four-ball, or equivalently if $-Y_0$ and $Y_1$ cobound a spin rational homology $S^3\times [0,1]$.

Given a marked oriented link $L$ with nonzero determinant, let $\spincs_L$ denote the spin structure on $\Sigma_2(S^3,L)$ determined by the orientation of $L$ as in the proof of Proposition \ref{prop:spin}.
It is not hard to see that $\spincs_{L\# L'}=\spincs_L\#\spincs_{L'}$.  If marked oriented links $L$ and $L'$ with nonzero determinant are $\chi$-concordant then by Propositions \ref{prop:2covers} and \ref{prop:spin} we see that
$(\Sigma_2(S^3,L),\spincs_L)$ is spin rational homology cobordant to $(\Sigma_2(S^3,L'),\spincs_{L'})$.
Thus taking $L$ to  the pair $(\Sigma_2(S^3,L),\spincs_L)$ gives a group homomorphism
$$\tcalf:\tcaln\to\Theta^3_{\qq,\spin}$$
from the subgroup of $\tlc$ represented by links with nonzero determinant to the spin rational homology cobordism group of spin rational homology three-spheres.

For a marked oriented link $L$ with nonzero determinant we define
$$\delta([L])=4\,d\circ\tcalf([L])=4d(\Sigma_2(S^3,L),\spincs_L),$$
where $d$ is the correction term invariant of \ozsvath\ and \szabo\ \cite{os4}.  
This is a composition of homomorphisms and thus a homomorphism from $\tcaln$ to $\qq$.
For a knot $K$, this is double the concordance invariant studied in \cite{conc}; the basic properties of $\delta$ for links are established in a similar manner.  From \cite{kt} and Proposition \ref{prop:spin}, $(\Sigma_2(S^3,L),\spincs_L)$ is the boundary of the spin four-manifold  given as the double branched cover of $D^4$ along a Seifert surface for $L$; moreover the signature of this manifold is equal to the signature of $L$.
By \cite[Theorem 1.2]{os4}, it follows that $\delta(L)$ is an integer and is congruent to minus the signature of $L$ modulo 8.

\begin{lemma}
\label{lem:alt}
Let $L$ be a nonsplit oriented alternating link.  Then $\sigma(L)+\delta(L)=0$.
\end{lemma}
\begin{proof}
In the special case
that $L$ is an arborescent link associated to a plumbing graph with no bad vertices,
this follows from results of Saveliev \cite{saveliev} and Stipsicz \cite{stipsicz}, each of whom show that
one of $\sigma(L)$, $-\delta(L)$ is equal to the Neumann-Siebenmann $\bar\mu$-invariant of the plumbing tree.

We follow the proof of \cite[Theorem 1.2]{conc} which establishes the result for alternating knots.
One may use the negative-definite Goeritz matrix $G$ of an alternating diagram for $L$ to compute
the signature, by a theorem of Gordon-Litherland \cite{gordonlitherland}, and also to compute the correction terms
of the double branched cover by results of \ozsvath-\szabo\ \cite[Proposition 3.2]{osu1}.

The proof given in \cite{conc} may be adapted virtually without change to the case of an oriented alternating
link and leads to the conclusion
$$\sigma(L)+4d(\Sigma_2(S^3,L),\spincs'_L)=0,$$
for \emph{some} spin structure $\spincs'_L$ on the double branched cover.  We will describe $\spincs'_L$ in terms
of a 4-manifold bounded by $\Sigma_2(S^3,L)$ and confirm that it is $\spincs_L$.

Choose an alternating diagram of $L$ and colour the complementary regions black and white in chessboard fashion, with
white regions to the left of the overpass after crossings as shown in Figure~\ref{fig:colour}.  Let $X$ denote the double cover
of $D^4$ branched along the properly embedded surface obtained by pushing the interior of the black surface into the interior
of $D^4$.  There is a simple procedure (cf.\ \cite{osu1}) for obtaining a Kirby diagram of $X$ from the given diagram of $L$: each crossing is
replaced by a clasp as in Figure~\ref{fig:Xw}, resulting in a link with one component for each white region.  The framing on each link is minus the number of
crossings adjacent to the corresponding white region.
These are the two-handles of $X$; we then add a single three-handle.  Choose one white region at random and label it $r_0$.
The two-handle corresponding to this region may be slid off the other two-handles and cancelled with the three-handle.
The intersection form in the basis given by the remaining two-handles is given by the Goeritz matrix of the diagram,
with $r_0$ as the ``region at infinity''.  We note each two-handle is attached along an unknot and therefore there is a two-sphere in
$X$ obtained by gluing the core of the two handle to a disk in $D^4$ bounded by the attaching circle.

The spin structure $\spincs'_L$ is described by a characteristic sublink of this diagram.
Given any two white regions in the alternating diagram one may connect them by a path consisting of
crossings in the diagram.
The orientation of the link determines a subset $S$ of the white regions
as follows: a region $r$ is in $S$ if there is a path from $r_0$ to $r$ using an odd number of negative crossings (and any number of positive crossings).
This in turn determines a sublink $C$ of the Kirby diagram for $X$ consisting of the components corresponding to regions in $S$.  It is easy to verify (see
\cite{conc}) that this is a characteristic sublink, or in other words if we let $\Sigma$ be the union of two-spheres in $X$ corresponding to the components of
$C$ then there is a spin structure on $X\setminus\Sigma$ which does not extend over $\Sigma$, i.e.\ which restricts to the trivial spin structure on any meridian of $\Sigma$.  The restriction of this spin structure
to the boundary of $X$ is $\spincs'_L$.

To verify that $\spincs'_L=\spincs_L$ we compare their restrictions to lifts of sums of two meridians of $L$.  It suffices to consider curves such as $\lambda$ in Figure \ref{fig:Xw}
which link two adjacent white regions $r,r'$.  This lifts to $\widetilde{\lambda}$ in the Kirby diagram for $X$, shown also in Figure \ref{fig:Xw}.  It follows that $\spincs'_L$ restricts to the trivial spin structure on
$\widetilde{\lambda}$
if and only if $r$ and $r'$ are both in or both not in $S$.  This is in turn equivalent to $h_L$ (as in the proof of Proposition \ref{prop:spin}) being nonzero, and $\spincs_L$ having trivial restriction, on $\widetilde{\lambda}$.
\end{proof}

\begin{figure}[htbp] 
{\centering
\ifpic
\includegraphics[width=5cm]{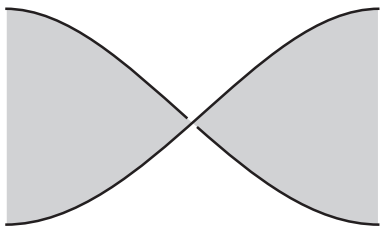}
\else \vskip 5cm \fi
\begin{narrow}{0.3in}{0.3in}
\caption{
\bf{Colouring convention for alternating diagrams.}}
\label{fig:colour}
\end{narrow}
}
\end{figure}

\begin{figure}[htbp]
{\centering
\ifpic
\psfrag{f}{\scriptsize{$-2$}}
\psfrag{a}{$\leadsto$}
\psfrag{u}{$\cup\text{ 3-handle}$}
\psfrag{l}{\scriptsize{$\lambda$}}
\psfrag{t}{\scriptsize{$\widetilde{\lambda}$}}
\includegraphics[width=11.5cm]{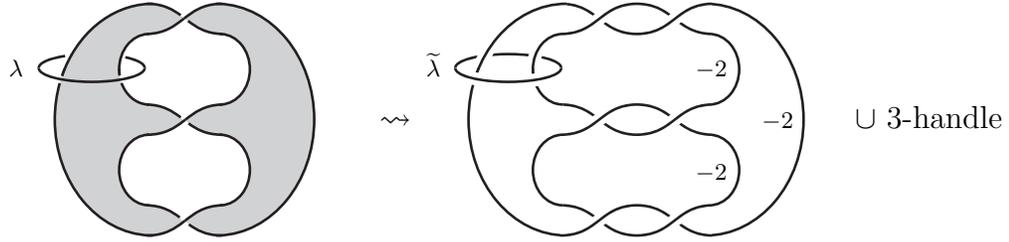}
\else \vskip 5cm \fi
\begin{narrow}{0.3in}{0.3in}
\caption{
\bf{A link diagram and the  double cover of $D^4$ branched along the black surface, showing the preimage of a curve $\lambda$.}}
\label{fig:Xw}
\end{narrow}
}
\end{figure}

We will use the homomorphisms $\tilde{l}$, $\sigma$ and $\delta$ to exhibit a direct summand of $\tcaln_0$ (this is the subgroup of $\tlc$ represented
by links with nonzero determinant and slice marked component).

The marked oriented links $\tilde{H}$ and $\tilde{L}_1$ from Figure \ref{fig:HL1} have
\begin{align*}
(\tilde{l},\sigma,\delta)(\tilde{H})&=(1,-1,1)\\
(\tilde{l},\sigma,\delta)(\tilde{L}_1)&=(1,0,0).
\end{align*}

Let $L_4$ be the Montesinos link (see for example \cite{conc}) given by plumbing twisted bands according to the positive-definite plumbing graph shown in Figure \ref{fig:L2graph}.  Each of its three components is an unknot and its determinant is 4.  Its signatures and the correction terms of its double branched cover may be computed using the plumbing graph (\cite{saveliev}, \cite{os6}); these turn out to be
\begin{align*}
\sigma&=\{-8,0,0,4\}\\
\delta&=\{0,0,0,-4\}.
\end{align*}
It follows that $\tilde{H}$, $\tilde{L}_1$ and  $L_4$ (with some choice of orientation) generate a $\zz^3$ summand of the direct complement $\tcaln_0$ of $\calc$ in $\tcaln$.  

\begin{figure}[htbp]
{\centering
\ifpic
\psfrag{1}{\scriptsize{$1$}}
\psfrag{2}{\scriptsize{$2$}}
\psfrag{4}{\scriptsize{$4$}}
\psfrag{6}{\scriptsize{$6$}}
\includegraphics[height=4cm]{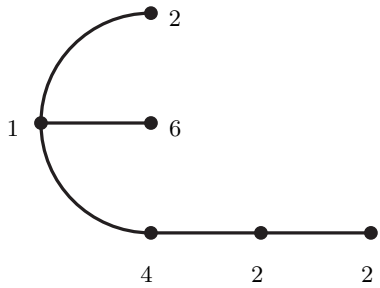}
\else \vskip 5cm \fi
\begin{narrow}{0.3in}{0.3in}
\caption{
\bf{The plumbing diagram for the Montesinos link $L_4$.}}
\label{fig:L2graph}
\end{narrow}
}
\end{figure}

A $\zz^\infty$ subgroup of $\tlc_0$ is given by a marked oriented version of that previously exhibited in $\lc_0$, i.e.\ it is generated by marked oriented
two-bridge links
$$\{S(2k,1)\,|\,k>2\}.$$  For an alternative argument that these links are independent in $\tlc$, using Levine-Tristram signatures, see the proof of Theorem \ref{thm:ortop}.

This completes the proof of Theorem \ref{thm:or}.

\begin{remark}
\label{rem:less2tors}
We note that the two-bridge links $S(q^2+1,q)$ which were shown in Proposition \ref{prop:2torsgroup} to generate $(\zz/2)^\infty<\lc_0$ have nonzero $\tilde{l}$ and hence infinite order in 
$\tlc$.
\end{remark}

\begin{remark}
\label{rem:inv}
The homomorphisms $\sigma$ and $\delta$ do not directly ``see'' the marked component of a marked oriented link but can give information about it nonetheless.
There is an involution $\rho$ on $\lc$ given by reversing orientation.
This has two lifts $\tilde{\rho}$ and $\tilde{\rho}'$ to $\tlc$: the former by reversing the orientation of the marked component
and the latter by reversing the orientation on all components.  A necessary condition for a marked oriented link $L$ to be trivial in $\tlc$ is for $\sigma$ and $\delta$
to vanish on $L$ and also on $\tilde{\rho}(L)$.  This can be used to show which component of the connected sum of Hopf and Whitehead links (see Figure \ref{fig:ribbonF}) may be marked
for that link to be trivial in $\tlc$ (though in this example that is also determined by linking numbers, cf.\ Example \ref{eg:HW}).
\end{remark}

\begin{remark}
\label{rem:deltas}
One might also expect to obtain homomorphisms
$\calf_{p^k}$ from some subgroup of $\tlc$
to the rational homology cobordism group of rational homology three-spheres, and possibly also to
the spin cobordism group $\Theta^3_{\qq,\spin}$, by taking prime power branched covers.
One could then extend Jabuka's homomorphisms \cite{jabuka}
$$\delta_{p^n}:\calc\to\zz$$
to a suitable subgroup of the link group $\tlc$.
\end{remark}

\begin{remark}
As we have seen, forgetting orientations on nonmarked components gives a surjection from marked oriented links to partly oriented links, inducing a surjection from $\tlc$ to $\lc$.  One could also forget which component is marked, giving a surjection from marked oriented links to oriented links.  We note this does \emph{not} induce a homomorphism from $\tlc$ to Hosokawa's link concordance group 
$\calh$ \cite{hosokawa}.  For example the rightmost link in Figure \ref{fig:ribbonF} is trivial in $\tlc$ (with appropriate choice of orientation and marked component), however it is nontrivial in Hosokawa's group with any chosen orientation since the sum of its pairwise linking numbers is $\pm1$ (it is shown in \cite{hosokawa} that this sum gives a surjection onto the direct complement of $\calc$ in Hosokawa's group).
\end{remark}

\section{Using locally flat surfaces}
\label{sec:top}
Replacing the word \emph{smooth} with \emph{locally flat} in Definition \ref{def:chiconc} leads to topological link concordance groups $\lctop$ and $\tlctop$.  We have the following topological versions of Theorem \ref{thm:mainthm} and Theorem \ref{thm:or}:

\begin{theorem}
\label{thm:lctop}
The set of locally flat $\chi$-concordance classes of partly oriented links forms an abelian group
$$\lctop\cong\calc_\Top\oplus(\lctop)_0$$
under connected sum which contains the topological knot concordance group $\calctop$ as a direct summand
(with  $\calctop\hookrightarrow\lctop$ induced by the inclusion of oriented knots into partly oriented links).

The complement $(\lctop)_0$ of $\calc_\Top$ in $\lctop$  contains a $(\zz/2)^\infty$ subgroup.
\end{theorem}

\begin{theorem}
\label{thm:ortop}
The set of locally flat $\chi$-concordance classes of marked oriented links forms an abelian group
$$\tlctop\cong\calctop\oplus(\tlctop)_0$$
under connected sum which contains the topological knot concordance group $\calctop$ as a direct summand
(with  $\calctop\hookrightarrow\tlctop$ induced by the inclusion of oriented knots into marked oriented links).
Forgetting orientations on nonmarked components induces a surjection $\tlctop\to\lctop$.

The complement $(\tlctop)_0$ of $\calc_\Top$ in $\tlctop$  contains a $\zz/2$ direct summand and a $\zz^\infty$ subgroup.
\end{theorem}

\begin{proof}[Proof of Theorem \ref{thm:lctop}]  Most of the proof is the same as that of Theorem \ref{thm:mainthm}, in particular Lemmas \ref{lem:equiv} and \ref{lem:groupsum} apply without modification.  Proposition \ref{prop:2covers} gives us a topological version of the branched double cover homomorphism
$\calf$.  Proposition \ref{prop:2torsgroup} shows that the two-bridge links $\{S(q^2+1,q)\}$ generate a $(\zz/2)^\infty$ subgroup in $(\lctop)_0$.
\end{proof}

One could presumably  reprove Lemma \ref{lem:F1F2}  using a mod 2 count of intersections of locally flat surfaces and hence recover a $\zz/2$ summand of $(\lctop)_0$ as in the smooth case.
One can also show using linking forms and results from \cite{short} that the two-bridge links $\{S(2k,1)\,|\,k\equiv 3\pmod4\}$ generate an infinitely generated subgroup of $(\lctop)_0$
consisting of elements of order at least 4.

\begin{proof}[Proof of  Theorem \ref{thm:ortop}] This largely follows the proof of Theorem \ref{thm:or} but this time we make use of Levine-Tristram signatures to establish that
the two-bridge links 
$$\{S(2k,1)\,|\,k>0\},$$
oriented so that the linking number is $+k$,  are linearly independent in $\tlctop$.

The Levine-Tristram signatures of these links are computed by Przytycki in \cite[Example 6.17]{prz}.  In particular $\sigma_\omega(S(2k,1))$ is a locally constant function of $\omega$ and changes at each $2k$th root of $-1$.  Suppose $\sum_{i=1}^n a_i\sigma(S(2k_i,1))=0$ for some integers $a_i$, with $0<k_1<\cdots<k_n$ and $a_n\ne0$.  For  $\omega=\exp(it)$ with $t\in[\pi-\pi/2k_{n-1},\pi-\pi/2k_n]$ we find
$$\sum_{i=1}^n a_i\sigma_\omega(S(2k_i,1))=a_n(\sigma_\omega(S(2k_n,1))-\sigma(S(2k_n,1)))\ne0.$$
Linear independence in $\tlctop$ then follows from Lemma \ref{lem:sig}, which also holds in the locally flat case.
\end{proof}

\begin{proof}[Proof of Theorem \ref{thm:TOPeg}]
Each of the links shown in Figure \ref{fig:DiffTop} is a connected sum of a partly oriented link $L_i$ and the Hopf link $H$, and each $L_i$ is a 2-component link with the same linking number as the  Hopf link.

Suppose that the partly oriented link $L_i\# H$ is (smoothly) $\chi$-nullconcordant.  Thus it bounds a smoothly embedded surface $F$ in $D^4$  which is either one disk and two \mobius\ bands, or a disk and an annulus, in each case with the marked component bounding the disk.  The first  possibility is ruled out by linking numbers as in Lemma \ref{lem:F1F2}, and the second  is equivalent to existence of a concordance in the traditional sense,
given by two properly embedded annuli in $S^3\times I$, between $L_i$ and $H$.  This is ruled out in the case of $L_3$ since $\delta(C)\ne0$ implies $C$ is not slice, and is ruled out in the case of $L_2$ by recent work of Cha-Kim-Ruberman-Strle \cite{ckrs}.

Each of $L_2$ and $L_3$ has Alexander polynomial one (\cite {ckrs}) and hence is locally flatly concordant (in the traditional sense and hence also $\chi$-concordant) to the Hopf link by a theorem of Davis \cite{davis}, from which it follows that $L_i\# H$ is trivial in $\lctop$ and (with a choice of orientation) in $\tlctop$.
\end{proof}


\section{Double branched covers of the four-ball}
\label{sec:2covers}

The following is a slight generalisation of Proposition \ref{prop:2covers}.

\begin{proposition}
\label{prop:2covers2}
Let $F$ be a locally flat properly embedded surface in $D^4$ with no closed components and Euler characteristic $n$. Suppose that the boundary of $F$ is a link $L$ with non-zero determinant. Then the double cover of $D^4$ branched along $F$ has $b_1=b_3=0$ and $b_2=1-n$.
\end{proposition}
Note that the surface here does not have to be connected or oriented. In the case where $F$ is a ribbon surface with $n=1$ this is proved in \cite{lisca1,lisca2}.  For smoothly embedded $F$ one could appeal to \cite{lw}.  It follows from Proposition \ref{prop:2covers2} that the \emph{slice Euler characteristic} $\chi_s(L)$ of a link with nonzero determinant is bounded above by 1 (presumably this is well-known).  Here $\chi_s(L)$ is the maximal Euler characteristic of a (smooth) surface $F$ as in Proposition \ref{prop:2covers2}.

\begin{proof}
The general strategy of the proof follows that of \cite[Theorem 3.6]{kt}.
Let $N = \Sigma_2(D^4 , F)$ be the double cover of $D^4$ branched along $F$.
We will construct $N$ by taking a double cover of $D^4 \setminus \nu F$, using a Gysin sequence to compute the homology, before regluing a copy of $\nu F$.
We use $\zz/2$ coefficients throughout.

The pair $(D^4 , S^3)$ can be decomposed as $(D^4 \setminus F \,\cup\, \nu F, S^3 \setminus L \,\cup\, L \times D^2)$.
Applying the relative Gysin and Mayer-Vietoris sequences gives an isomorphism
\begin{equation} \label{reliso}H^1(\partial\nu F, L \times S^1) \cong H^1(F , L),\end{equation}
and also
$$H^1(D^4 \setminus F, S^3 \setminus L) =0.$$
In addition the isomorphism in \eqref{reliso} is induced by the inclusion of $\partial\nu F$ into $\nu F$.

The relative Gysin sequence \cite[Theorem 11.7.36]{atfhv} can also be applied to the pair $(D^4 \setminus F, S^3 \setminus L)$ with  the real line bundle associated to the double cover.
The relevant part of the Gysin sequence is
$$H^1(D^4 \setminus F, S^3 \setminus L) \rightarrow H^{1}(\widetilde{D^4 \setminus F} , \widetilde{S^3 \setminus L}) \rightarrow H^1(D^4 \setminus F, S^3 \setminus L),$$
whence $H^{1}(\widetilde{D^4 \setminus F} , \widetilde{S^3 \setminus L})=0$.

This can be used to calculate the Betti numbers of $N$, which is constructed from the double cover of $D^4 \setminus F$ by attaching $D^2 \times F$.
Applying the Mayer-Vietoris sequence again gives
$$0 \rightarrow H^1(N,\partial N) \rightarrow H^1 (\widetilde{D^4 \setminus F} , \widetilde{S^3 \setminus L}) \oplus H^1(F,L) \rightarrow H^1(\partial\nu F, L \times S^1) \rightarrow \ldots .$$
Combining this with \eqref{reliso} we see that $H^1(N,\partial N)=0$.
Since $N$ is compact and orientable with rational homology sphere boundary, we have
$$b_1(N)=b_3(N)=0.$$
The Euler characteristic of $N$ is given by $\chi(N) = 2\chi(D^4) - \chi(F) = 2-n$,
from which we see that $b_2(N)=1-n$.
\end{proof}

\begin{proof}[Proof of Corollary \ref{cor:2bridge}]  Assume $p$ is even, since the odd case was established in \cite{lisca1}.  In order to have the correct Euler characteristic and number of boundary components, $F$ must be the union of a disk and a M\"obius band. By Proposition \ref{prop:2covers2}, the double cover of $D^4$ branched over $F$ is a rational homology ball and is bounded by the lens space $L(p,q)$.
By a result of Lisca \cite[Theorem 1.2]{lisca1}, there is a ribbon embedding of $F$ in  $D^4$.
\end{proof}

In the wake of Lisca's work, the slice-ribbon conjecture was established by Greene and Jabuka for three-strand pretzel knots $P(a,b,c)$ with $a, b, c$ odd \cite{gj}, and by Lecuona for a different family of 3-tangle Montesinos knots \cite{lecuona}.  It seems likely that their methods may  be combined with Proposition \ref{prop:2covers} to prove a statement analagous to Corollary \ref{cor:2bridge} for some 3-tangle Montesinos links.


\section{Quotients of monoids with involution}
\label{sec:alg}
The reader may have noticed that we have made use of various topological obstructions to a link being $\chi$-slice, that is being the boundary of a properly embedded surface $F$ of Euler characteristic one in $D^4$.  Most of these obstructions do not take account of the marked component.  One may ask, why not simply take the quotient of links by $\chi$-slice links?

The first point to note is that in order for connected sum using marked components to be well-defined on the quotient (see Lemma \ref{lem:groupsum}), we need to specify that the marked component 
of a $\chi$-slice link is the boundary of a disk component of $F$.  The second point regards transitivity of the $\chi$-concordance relation.  This may be understood in terms of the following simple lemma about monoids with involution. 

\begin{lemma}
\label{lem:alg}
Let $(A,\#,-)$ be a commutative  monoid with involution and let $B\subset A$ be a submonoid closed under $-$.  Let $A/B$ denote the quotient of $A$ by
$$a_1\sim a_2\iff a_1\# b_1=a_2\# b_2,\qquad \text{for some } b_i\in B.$$
Then $A/B$ is an abelian group with $[a]^{-1}=[-a]$, and the equivalence relation can be rewritten as
$$a_1\sim a_2\iff -a_1\#a_2\in B,$$
if the following conditions hold:
\begin{enumerate}
\item $-a\#a\in B$, for all $a\in A$, and\label{B1}
\item $a\#b,b\in B\implies a\in B$.\label{B2}
\end{enumerate}
Conversely, given any  morphism of monoids with involution from $A$ to a group, the kernel is a submonoid with involution satisfying \eqref{B1} and \eqref{B2}.
\end{lemma}
\begin{proof}
Straightforward exercise.
\end{proof}

For example, one may take $A$ to be oriented knots, with $B$ given by slice knots.  One may also take $A$ to be partly oriented or marked oriented links, with $B$ given by $\chi$-nullconcordant links as in Definition \ref{def:chiconc}.  As the following lemma shows, the smallest submonoid satisfying the conditions of Lemma \ref{lem:alg} which contains $\chi$-slice links is ``too large", in that the resulting quotient would be an uninteresting extension of the knot concordance group.
\begin{lemma}
Let $L$ be a partly oriented link with marked oriented component $K$, and let $H$ be the Hopf link.  Then for some $l\in\{0,1\}$ and some unlink $U$, $L\#-\!K\#lH\#U$ is $\chi$-slice.  The same conclusion holds with $l\in\zz$ if $L$ is a marked oriented link.
\end{lemma}
\begin{proof}
The marked component (call it $K'$) of $L'=L\#-\!K\#lH$ is slice and has linking number zero (respectively even) with $L'\setminus K'$, if $l=-\lk(L,K)$ (resp., modulo 2).
Let $\Delta$ be a slice disk, and let $F$ be a smoothly properly embedded orientable surface bounded by $L'\setminus K'$, intersecting $\Delta$ transversely with algebraic intersection number zero (resp., even).  Adding handles to $F$ to remove intersection points in pairs results in an orientable (resp., a possibly nonorientable) embedded surface $F'$ bounded by $L'$ with Euler characteristic $m$, which we may assume to be negative after adding some extra handles.  Connect summing $L'$ with an unlink, and boundary summing $F'$ with a union of disks, gives the result.
\end{proof}


\section{Open questions}
\label{sec:questions}

Here are a few questions that seem interesting to the authors.
\begin{itemize}
\item What are the orders of the Whitehead link and the Borromean rings, in $\lc$ or $\tlc$?  (Their branched double covers have orders 2 and 1 in $\Theta^3_\qq$.)
\item Is there any interesting torsion (not resulting from negative amphichiral links) in either $\lc$ or $\tlc$?
\item Do the Rasmussen $s$ \cite{rasmussen}, see also \cite{bw,lewark}, and \ozsvath-\szabo\ $\tau$ \cite{ostau} invariants give homomorphisms from $\tcaln$ to $\zz$?
\end{itemize}

\end{document}